\documentclass[11pt]{amsart}
\usepackage{amsfonts}
\usepackage{graphicx}
\usepackage{tabularx}
\usepackage{array}
\usepackage[usenames,dvipsnames]{color}
\usepackage{comment}
\usepackage{amsmath}
\usepackage{amsthm}
\usepackage{amssymb}
\usepackage{fullpage}
\usepackage[dvipsnames]{xcolor}
\usepackage{listings}
\usepackage{pgfplots}% loads also tikz
\pgfplotsset{compat=newest}% to avoid the pgfplots warning
\usetikzlibrary{intersections, pgfplots.fillbetween}
\pgfdeclarelayer{pre main}
\usetikzlibrary{decorations.markings}

\tikzstyle{legend_general}=[rectangle, rounded corners, thin,
                          top color= white,bottom color=lavander!25,
                          minimum width=2.5cm, minimum height=0.8cm,
                          violet]

\renewcommand{\epsilon}{\varepsilon}

\DeclareMathOperator{\HG}{HG}

\DeclareMathOperator{\ex}{\mathcal{E}}

\newtheorem{theorem}{Theorem}[section]

\newtheorem{question}[theorem]{Question}
\newtheorem{claim}[theorem]{Claim}
\newtheorem{lemma}[theorem]{Lemma}

\theoremstyle{definition}
\newtheorem{definition}[theorem]{Definition}

\author{Peter Bradshaw}
\address{Department of Mathematics, Simon Fraser University, Vancouver, Canada}
\email{pabradsh@sfu.ca}

\title{On the hat guessing number of a planar graph class}
\begin{document}
\maketitle
\begin{abstract}
The \emph{hat guessing number} is a graph invariant based on a hat guessing game introduced by Winkler.
Using a new vertex decomposition argument involving an edge density theorem of Erd\H{o}s for hypergraphs, we show that the hat guessing number of all outerplanar graphs is less than $2^{125000}$. We also define the class of \emph{layered planar graphs}, which contains outerplanar graphs, and we show that every layered planar graph has bounded hat guessing number.
\end{abstract}
\section{Introduction}
The \emph{hat guessing game} is defined as follows. We have a graph $G$, and a \emph{player} resides at each vertex of $G$. For each vertex $v \in V(G)$, the player at $v$ can see exactly those players at the neighbors of $v$. In particular, a player cannot see himself. An \emph{adversary} possesses a large collection of hats of different colors. When the game starts, the adversary places a hat on the head of each player, and then each player privately guesses the color of his hat. The players win the game if at least one player correctly guesses the color of his hat; otherwise, the adversary wins. Before the game begins, the players may come together to devise a guessing strategy, but this strategy is known to the adversary, and the adversary may choose a hat assignment with the strategy of the players in mind. The hat guessing game was first considered for complete graphs by Winkler \cite{Winkler} and later for general graphs by Butler, Hajiaghayi, Kleinberg, and Leighton \cite{Butler}.

The hat guessing game is typically studied with the following two assumptions. First, it is assumed that the adversary possesses enough hats of each color so that no color will ever run out while the adversary is assigning hats to players. Second, it is assumed that the players follow a deterministic strategy to guess their hat colors; that is, the guess of a player at a vertex $v$ is uniquely determined by the hat colors at neighbors of $v$. Given a graph $G$, we say that the \emph{hat guessing number} of $G$ is the maximum number $k$ of hat colors such that the players on $G$ have a strategy that guarantees that at least one player will correctly guess his hat color when each player is given a hat with a color from the set $\{1, \dots, k\}$. We write $\HG(G)$ for the hat guessing number of $G$. In other words, if $\HG(G) \geq k$, then there exists a strategy for players on the graph $G$ such that for any hat color assignment $V(G) \rightarrow \{1, \dots, k\}$, at least one player will correctly guess the color of his hat.

We may describe the hat guessing game formally as a graph coloring problem as follows. Let $G$ be a graph, and let $S = \{1, \dots, k\}$ be a set of colors. We define a \emph{hat guessing strategy} on $G$ to be a family $\Gamma = \{\Gamma_v\}_{v \in V(G)}$ of functions, where each function is a mapping $\Gamma_v: S^{N(v)} \rightarrow S$; that is, each function $\Gamma_v$ takes a coloring of $N(v)$ as input and returns a color from $S$ as output. We say that the strategy $\Gamma$ is a \emph{winning strategy} if, for every (not necessarily proper) graph coloring $\phi:V(G) \rightarrow S$, there exists a vertex $v \in V(G)$ with neighbors $(u_1, \dots, u_t)$ such that $\Gamma_v$ maps $(\phi(u_1), \dots, \phi(u_t))$ to $\phi(v)$. It is clear that the players win the hat guessing game on $G$ with hat color set $S$ if and only if there exists a winning strategy $\Gamma$ on $G$. It was shown in \cite{KokhasCliquesAndConstructorsI} that with optimal play, the winner of the hat guessing game does not change even when each vertex $v \in V(G)$ receives a hat from an arbitrary list $L_v$ of size $|S|$, rather than from the set $S$. Furthermore, it was shown that if each vertex $v \in V(G)$ has a list $L_v$ of possible hat colors, then only the size of each list $L_v$ affects whether or not the players have a winning strategy on $G$.

We find it useful to give a simple example of a winning strategy in the hat guessing game. We consider the hat guessing game played on $K_2$ in which the adversary has only red hats and blue hats. We refer to the players as Alice and Bob. It is straightforward to see that Alice and Bob have the following winning strategy: Alice will guess the color of Bob's hat, and Bob will guess the opposite color of Alice's hat. This way, if Alice and Bob receive the same hat color, then Alice will guess correctly. On the other hand, if Alice and Bob receive different hat colors, then Bob will guess correctly. This shows that $\HG(K_2) \geq 2$, and in fact, $\HG(K_2) = 2$ \cite{Feige}.

While the rules of the hat guessing game are simple, establishing upper bounds for the hat guessing numbers achieved by graphs in large classes is surprisingly difficult.
In contrast to other game-related graph parameters such as game coloring number \cite{Zhu} and cop number \cite{Joret, Bowler}, no upper bound is known for the hat guessing number of graphs of bounded treewidth or graphs of bounded genus. Farnik \cite{Farnik} asked whether the hat guessing number of a graph $G$ is bounded by some function of the degeneracy of $G$, but this question remains unanswered, and graphs of degeneracy $d$ and hat guessing number at least $2^{2^{d-1}}$ have been constructed \cite{HeDegeneracy}.
In particular, the following more specific question, appearing in \cite{Jaroslaw} and \cite{Kokhas2021}, remains unanswered:
\begin{question}
\label{ques:main}
Is the hat guessing number of planar graphs bounded above by some universal constant?
\end{question}
While Question \ref{ques:main} has not yet been answered, Kokhas and Latyshev \cite{Kokhas2021} have shown examples of planar graphs $G$ for which $\HG(G) \geq 14$.
For some restricted graph classes, such as cliques \cite{Feige}, cycles \cite{Szcz}, complete bipartite graphs \cite{Alon, GadouleauHats}, graphs of bounded degree \cite{Farnik}, graphs of bounded treedepth \cite{HeDegeneracy}, embedded graphs of large girth \cite{Jaroslaw}, cliques joined at a single cut-vertex, and split graphs \cite{HeWindmill}, bounds for the hat guessing number have been determined.

In this paper, we will take a major step toward answering Question \ref{ques:main} by proving an affirmative answer for a large class of planar graphs. For outerplanar graphs, we will prove the following result.
\begin{theorem}
\label{thm:OP_intro}
If $G$ is an outerplanar graph, then $\HG(G) < 2^{125000}$.
\end{theorem}
We will also prove an upper bound for the hat guessing number of \emph{layered planar graphs}, which we roughly define as planar graphs that can be obtained by beginning with a $2$-connected outerplanar graph $G_1$, and then for $1 \leq i \leq \tau$, adding a $2$-connected outerplanar graph $G_{i+1}$ to some interior face of $G_i$ and adding non-crossing edges between $G_i$ and $G_{i+1}$. We also use the term \emph{layered planar graph} to describe a subgraph of such a planar graph. 
For layered planar graphs, we have the following result.

\begin{theorem}
\label{thm:planar_intro}
If $G$ is a layered planar graph, then $\log_2 \log_2 \log _2 \log_2 \log_2 \HG(G) < 149$.
\end{theorem}

Theorems \ref{thm:OP_intro} and \ref{thm:planar_intro} are the first results that show upper bounds for the hat guessing number of large topologically defined graph classes.
The main ingredients for the proofs of Theorems \ref{thm:OP_intro} and \ref{thm:planar_intro} will be a vertex partition lemma from Bosek et al.~\cite{Jaroslaw}, as well as a new theorem that bounds the hat guessing number of graphs that admit a vertex partition with a certain tree-like structure. The proof of this new theorem uses an argument based on a Tur\'an-type edge density problem, and the combination of our edge-density argument and the lemma of Bosek et al.~is what causes our upper bounds to be so large.

The paper will be organized as follows. In Section \ref{sec:Tools}, we will introduce some important tools that we will need for our two main theorems. In Section \ref{sec:OP}, we will show that every outerplanar graph has a vertex partition satisfying certain key properties, and then with the help of the tools introduced in Section \ref{sec:Tools}, we will prove Theorem \ref{thm:OP_intro}. In Section \ref{sec:planar}, we will use a similar strategy to extend our methods beyond outerplanar graphs and prove Theorem \ref{thm:planar_intro}. Finally, in Section \ref{sec:genus}, we will show that if an upper bound can be obtained for a certain stronger version of the hat guessing number on planar graphs, then an upper bound on the hat guessing number can be obtained for all graphs of bounded genus.

\section{Multiple guesses, vertex partitions, and edge density}
\label{sec:Tools}
In this section, we will outline three key tools that we will use to prove Theorem \ref{thm:OP_intro} and \ref{thm:planar_intro}.
These tools use a modified version of the hat guessing game in which each player attempts to guess his hat color $s$ times (without hearing the guesses of the other players). Given a graph $G$ and an integer $s \geq 1$, if $k$ is the maximum integer for which the players on $G$, when assigned hats from the color set $\{1, \dots, k\}$, have a strategy that guarantees at least one correct hat color guess when each player is allowed to guess $s$ times, then we write $\HG_s(G) = k$. 

The first of our tools follows from a simple application of Lov\'asz's Local Lemma \cite{LLL}; see \cite{Farnik} for more details.
\begin{lemma}
\label{lem:LLL}
Let $s \geq 1$ be an integer. If $G$ is a graph of maximum degree $\Delta$, then $\HG_s(G) < (\Delta+1)es$.
\end{lemma}

Our next tool tells us that if the vertices of a graph can be partitioned into sets satisfying certain conditions, then the hat guessing number of the graph is bounded. This tool was introduced by Bosek et al.~\cite{Jaroslaw}.
%, but since the proof is simple and elegant, we will include it.

\begin{lemma}[\cite{Jaroslaw}]
\label{lem:partition}
Let $s \geq 1$ be an integer. Let $G$ be a graph, and let $V(G) = A \cup B$ be a partition of the vertices of $G$. If each vertex in $A$ has at most $d$ neighbors in $B$, then $\HG_s(G) \leq \HG_{s'}(G[A])$, where $s' = s(\HG_s(G[B]) + 1)^d$.
\end{lemma}

By using the same approach originally used in \cite{Jaroslaw}, we can prove the following more general version of Lemma \ref{lem:partition}. Note that Lemma \ref{lem:partition} is obtained from the following lemma by setting $k = 2$, $V_1 = A$, and $V_2 = B$.

\begin{lemma}
\label{lem:gen_partition}
Let $s \geq 1$ be an integer.
Let $G$ be a graph with a vertex partition $V(G) = V_1 \cup \dots \cup V_k$, and let $\ell_1, \dots, \ell_k$ be positive integers. Assume that for each pair $i,j$ satisfying $1 \leq i < j \leq k$, each vertex $v \in V_i$ has at most $d_{i,j}$ neighbors in $V_j$. For $1 \leq i \leq k-1$, define
\[s_i = s \ell_{i+1}^{d_{i,i+1}} \cdot \ell_{i+2}^{d_{i,i+2}} \cdot \dots \cdot \ell_{k}^{d_{i,k}},\]
and define $s_k = s$. 
If, for each value $1 \leq i \leq k$, it holds that 
\[\HG_{s_i}(G[V_i]) < \ell_i,\]
then $\HG_s(G) <\max \{\ell_1, \dots, \ell_k\}$.
\end{lemma}
\begin{proof}
For $1 \leq i \leq k$, we fix a list of exactly $\ell_i$ colors at each vertex in $V_i$, and we consider the game in which each player makes $s$ guesses. We will prove the following statement:
\begin{quote}
For each value $1 \leq i \leq k$, if the hat colors on $V_1 \cup \dots \cup V_{i-1}$ are already fixed, then there exists a hat assignment on $V_i$ for which no vertex in $V_i$ correctly guesses its hat color, regardless of how $V_{i+1} \cup \dots \cup V_k$ is colored.
\end{quote}
By iteratively applying this statement for each value $1 \leq i \leq k$, we obtain a winning hat assignment on $G$ that uses at most $\max \{\ell_1, \dots, \ell_k\}$ colors, proving the lemma.

To prove this statement, consider a value $1 \leq i \leq k$, and assume that hat colors are fixed on $V_1 \cup \dots \cup V_{i-1}$. With these hat colors fixed, every vertex in $V_i$ has a hat guessing function depending only on $G[V_i \cup \dots \cup V_k]$. 
Furthermore, for each vertex $v \in V_i$, every possible coloring of $N(v) \cap (V_{i+1} \cup \dots \cup V_k)$ gives $v$ a unique guessing function depending only on $G[V_i]$, and there are $\ell_{i+1}^{d_{i,i+1}} \dots \ell_{k}^{d_{i,k}}$ possible colorings of $N(v) \cap (V_{i+1} \cup \dots \cup V_k)$. Therefore, with color lists fixed at every vertex of $V_{i+1} \cup \dots \cup V_k$, for each hat assignment on $G[V_i]$, $v$ will guess from a total of at most $s_i$ possible guesses. (Note that when $i = k$, the set $V_{i+1} \cup \dots \cup V_k$ is empty, so $v$ guesses from a total of $s_k = s$ guesses.) By our assumption, we may assign each vertex of $V_i$ a hat from its set of $\ell_i$ colors in such a way that no vertex of $G[V_i]$ guesses its hat color correctly, even with $s_i$ guesses. We give $G[V_i]$ such a hat assignment, and since each vertex $v \in V_i$ guesses from a set of at most $s_i$ colors, no vertex of $V_i$ guesses its hat color correctly. This completes the proof.
\end{proof}

Finally, we will define a third tool that we will need for Theorems \ref{thm:OP_intro} and \ref{thm:planar_intro}.
Our last tool relies heavily on theory related to a Tur\'an-type edge density problem. We will need some definitions. First, an \emph{$r$-partite $r$-uniform hypergraph} $\mathcal H$ is defined as a set $V$ of vertices and a collection $E$ of $r$-tuples from $V$, satisfying the following property: $V$ can be partitioned into $r$ parts $V_1, \dots, V_r$ so that every $r$-tuple in $E$ intersects each part $V_i$ at exactly one vertex. We often use the term \emph{$r$-partite $r$-graph} to refer to an $r$-partite $r$-uniform hypergraph, and we often call the $r$-tuples in $E$ \emph{edges}. 
We say that an $r$-partite $r$-graph is \emph{balanced} if $|V_1| = \dots = |V_r|$.
We say that an $r$-partite $r$-graph $\mathcal K$ is \emph{complete} if it contains every possible edge $e$ satisfying $|e \cap V_i| = 1$ for each vertex part $V_i$, and if $\mathcal K$ is also balanced and has $r \ell$ vertices, then we write $\mathcal K = K^{(r)}_{\ell}$.
Next, for integers $r \geq 1$, $\ell \geq r$, and $n \geq \ell$, we define the quantity $\ex^{(r)}(n,\ell)$ to be the maximum number of edges in a balanced $r$-uniform $r$-graph with $rn$ vertices and with no complete $K^{(r)}_{\ell}$ subgraph. Since a balanced $r$-uniform $r$-graph with $rn$ vertices and $n^r$ edges certainly must contain such a subgraph, we see that $\ex^{(r)}(n,\ell)$ is well-defined and less than $n^r$.

We give several examples of the quantity $\ex^{(r)}(n,\ell)$. When $r = 1$, an $r$-partite $r$-graph is simply a collection of vertices in which some of these vertices are also called edges, and a $K^{(1)}_{\ell}$ is simply a collection of $\ell$ of these ``edges." Any $1$-partite $1$-graph with at least $\ell$ edges clearly must contain a $K^{(1)}_{\ell}$, so for all $1 \leq \ell \leq n$, $\ex^{(1)}(n,\ell) = \ell - 1$.
When $r = 2$, the quantity $\ex^{(2)}(n,\ell)$ describes the maximum number of edges in a balanced bipartite graph on $2n$ vertices containing no copy of $K_{\ell,\ell}$. The question of determining the precise value of $\ex^{(2)}(n,\ell)$ is a special case of a classic problem of Zarankievicz, which asks how many edges a bipartite graph on $m+n$ vertices can have without containing a copy of $K_{s,t}$. This problem of Zarankiewicz has a long history and has led to many beautiful results; see \cite{Furedi} for an extensive survey of this area of combinatorics.

For our final tool, we will need the following theorem of Erd\H{o}s \cite{Erdos1964}, which gives an upper bound for $\ex^{(r)}(n,\ell)$. The original result of Erd\H{o}s assumes that $n$ is sufficiently large, but we will need a result that holds even for small $n$, so we present a slightly modified form of Erd\H{o}s's original result.

\begin{lemma}
\label{lem:Erdos1}
Let $r \geq 2$, let $\mathcal H$ be a balanced $r$-partite $r$-graph with $rn$ vertices, and let $2 \leq \ell \leq n$. If $|E(\mathcal H)| \geq 3n^{r - \frac{1}{\ell^{r-1}}}$, then $\mathcal H$ contains a copy of the $r$-graph $K^{(r)}_{\ell}$.
\end{lemma}
\begin{proof}
The proof of the lemma is very similar to the original proof of Erd\H{o}s, and we defer the proof to the appendix.
\end{proof}

Before we introduce our last tool for proving Theorem \ref{thm:OP}, we will need some more definitions. For a graph $G$ and a vertex subset $U \subseteq V(G)$, we write $N(U)$ for the set of vertices with at least one neighbor in $U$. Given a graph $G$ with a vertex partition $V_1, \dots V_t$, we define the \emph{quotient graph} $G / (V_1, \dots, V_t)$ as the graph obtained from $G$ by contracting each part $V_i$ and deleting loops and parallel edges. In other words, $G/(V_1, \dots ,V_t)$ has $t$ vertices $v_1, \dots, v_t$, and $v_i$ is adjacent to $v_j$ if and only if some edge in $G$ joins a vertex of $V_i$ to a vertex of $V_j$.

We are now ready for our last main tool for proving Theorem \ref{thm:OP}. 
The following result, which is useful for bounding the hat guessing number of outerplanar  and layered planar graphs, is also interesting in its own right. 
Our proof of this result uses key ideas from the proof of Butler et al.~\cite{Butler} that $\HG(T) = 2$ for every tree $T$ on at least two vertices.

\begin{theorem}
\label{thm:tree}
Let $r,s \geq 1$ be integers.
Let $G$ be a graph, and let $V_1, \dots, V_t$ be a partition of $V(G)$ such that the quotient graph $G/(V_1, \dots, V_t)$ is a tree. Suppose that for each distinct pair $V_i, V_j$, it holds that $|N(V_i) \cap V_j| \leq r$. If $\HG_s(G[V_i]) < \ell$ for all $1 \leq i \leq t$, then $\HG_s(G) \leq (3 \ell)^{r\ell^{r-1}}$ when $r \geq 2$, and $\HG_s(G) \leq \ell (\ell - 1)$ when $r = 1$.
\end{theorem}
\begin{proof}
We would like to assume that for each $V_i$, every neighboring set $V_j$ satisfies $|N(V_i) \cap V_j| = r$.
This can be achieved by adding isolated vertices to each neighboring set $V_j$ and then adding edges between these new vertices and vertices of $V_i$. These extra vertices will not cause any of our hypotheses to be violated, nor will they decrease the hat guessing number of $G$.

We let $k = (3\ell)^{r\ell^{r-1}} + 1$ when $r \geq 2$, and we let $k = \ell(\ell-1)+1$ when $r = 1$. We aim to show that $\HG_s(G) < k$.
We first make the following claim.
\begin{claim}
 $\ell^r \ex^{(r)}(k,\ell) < k^r$.
\end{claim}
\begin{proof}
When $r = 1$, the claim asserts that $\ell (\ell - 1) < k$, which is clearly true. When $r \geq 2$, Lemma \ref{lem:Erdos1} states that $\ex^{(r)}(k,\ell) < 3  k^{r - \frac{1}{\ell^{r-1}}}$. Then, $\ell^r \ex^{(r)}(k,\ell) < 3\ell ^r k^{-\frac{1}{\ell^{r-1}}} k^r < k^r$.
\end{proof}

Now, we fix a guessing strategy $\Gamma = \{\Gamma_v\}_{v \in V(G)}$ on $G$. We prove the following stronger statement.
\begin{quote}
Let $1 \leq i \leq t$. If every vertex in $V_i$ has a list of $\ell$ colors and every other vertex in $G$ has a list of $k$ colors, then the adversary has a winning hat assignment.
\end{quote}
We induct on $t$. When $t=1$, the statement holds from the fact that $\HG_s(G[V_1]) = \HG_s(G) < \ell$. Now, suppose $t > 1$, and let $i$ be some value satisfying $1 \leq i \leq t$. We must show that if every vertex in $V_i$ has a list of $\ell$ possible hat colors and every other vertex of $G$ has a list of $k$ possible hat colors, then the adversary has a winning hat assignment.

In each neighboring set $V_j$ of $V_i$, there exists a set $U_j \subseteq V_j$ of exactly $r$ vertices that have neighbors in $V_i$, and there also exists a set $W_j \subseteq V_i$ of exactly $r$ vertices that are neighbors of $U_j$. We write $C_j$ for the component of $G \setminus V_i$ containing $V_j$. 
 If a hat assignment $\alpha$ on $W_j$ is fixed, then $\Gamma$ determines a unique hat guessing strategy on $C_j$. Furthermore, by the induction hypothesis, if each vertex of $C_j$ has the color list $\{1, \dots, k\}$, then with $\alpha$ fixed, the adversary has a winning hat assignment on $C_j$. We let $B_{\alpha, j}$ be the set of hat assignments on $C_j$ that cause the adversary to win the restricted game on $C_j$ when the players use the hat guessing strategy determined by $\alpha$. Then, we let $A_{\alpha,j}$ be the set of distinct colorings of $U_j$ that can be obtained by restricting an assignment from $B_{\alpha,j}$ to $U_j$. We see from the induction hypothesis that $A_{\alpha,j}$ is nonempty. Now, we make the following claim:
\begin{claim}
 If $\alpha$ is a fixed hat assignment on $W_j$, then $A_{\alpha,j}$ contains at least $k^r - \ex^{(r)}(k,\ell)$ distinct colorings.
\end{claim}
\begin{proof}
%We will assume that $|U_j| = r$. If $|U_j| < r$, then we may use a similar method, and then the result will follow from the fact that $\ex^{(r)}(k,\ell)$ is increasing with respect to $r$.
%
Suppose that $A_{\alpha,j}$ contains at most $k^r - \ex^{(r)}(k,\ell) - 1$ distinct colorings. We construct a balanced $r$-partite $r$-graph $\mathcal H$ on $kr$ vertices as follows. 
We write $U_j = \{u_1, \dots, u_r\}$.
Then, we let the $kr$ vertices of $\mathcal H$ be indexed by $(u_p, q)$, where $u_p \in U_j$ and $1 \leq q \leq k$. Finally, for each hat assignment in $A_{\alpha,j}$ in which each vertex $u_p \in U_j$ is given a hat of some color $\gamma_p$, we add an edge to $\mathcal H$ of the form $\{(u_1, \gamma_1), (u_2, \gamma_2),  \dots, (u_r, \gamma_r)\}$. Now, since $A_{\alpha,j}$ has at most $k^r - \ex(k,r) - 1$ edges, it follows that the complement $r$-graph $\overline{\mathcal H}$ contains at least $\ex^{(r)}(k,\ell) + 1$ edges, and hence $\overline{\mathcal H}$ contains a copy of $K^{(r)}_{\ell}$. In other words, there exist sets $L_1 \subseteq \{1, \dots, k\}, \dots, L_r \subseteq \{1, \dots, k\}$, each of size $\ell$, such that $A_{\alpha,j}$ contains no hat assignment in which each vertex $u_p$ is assigned a hat with a color from $L_p$. It then follows that when each vertex $u_p \in U_j$ is given the color list $L_p$ and every other vertex of $C_j$ is given a list of $k$ colors, the adversary has no winning hat assignment on $C_j$ using these lists. However, this contradicts the induction hypothesis applied to $C_j$ and with $V_j$ instead of $V_i$. Thus, the claim holds.
\end{proof}
Now, for each component $C_j$, we compute $A_{\alpha,j}$ for each of the $\ell^r$ hat assignments $\alpha$ on $W_j$ using the pre-assigned lists of size $\ell$. Since $\ell^r \ex^{(r)}(k,\ell) < k^r$, for each $j$, the intersection $\bigcap_{\alpha} A_{\alpha, j}$ is nonempty by the pigeonhole principle, where $\alpha$ is taken over the $\ell^r$ hat assignments on $W_j$. Hence, for each $j$, we can choose an assignment $A_j$ from this intersection, and we use $A_j$ to assign hats to the vertices in $U_j$. 

Next, with hats already assigned to each $U_j$, the vertices in $V_i$ have a fixed guessing strategy that depends only on the hat assignment at $V_i$. Since $\HG_s(G[V_i]) < \ell$, we can assign each vertex in $V_i$ a hat in such a way that no vertex guesses its hat color correctly. 

Finally, with colors assigned to $V_i$, we argue that we can complete a hat assignment on each component $C_j$ in such a way that no vertex in $C_j$ guesses its hat color correctly. Since $V_i$ has already been colored, a coloring $\alpha$ on $W_j$ is fixed, and hence the adversary can give a winning assignment to $C_j$ if and only if some assignment $B_j \in B_{\alpha,j}$ extends the already fixed coloring $A_j$ at $U_j$. However, since $A_j \in A_{\alpha,j}$, the set of restricted assignments of $B_{\alpha,j}$, such a winning assignment $B_j$ must exist by definition. Therefore, we use $B_j$ to color $C_j$, and hence no vertex in $C_j$ will guess its hat color correctly. By repeating this process for each component $C_j$, we assign hats to all remaining in such a way that no vertex will guess its hat color correctly.
%Now, it remains to show that we can complete our hat assignment by assigning hats to $V_i$ in such a way that no remaining vertex guesses its hat color correctly. Since our hat assignments at $V_1, \dots, V_{i-1}, V_{i+1}, \dots, V_t$ are already determined, each vertex of the graph $G[V_i]$ has a guessing function depending only on $G[V_i]$. Since $\HG_s(G[V_i]) < \ell$ and the lists of possible hat colors at $V_i$ each have $\ell$ colors, the adversary has a winning hat assignment at $V_i$. This completes the proof.
\end{proof}

\section{Outerplanar graphs}
\label{sec:OP}
In this section, we prove Theorem \ref{thm:OP_intro}, showing that the hat guessing number of outerplanar graphs is bounded. We need some definitions and lemmas.
\begin{definition}
We define a \emph{petal graph} $G$ to be a graph obtained from a (possibly empty) path $P$ by adding a vertex $v$ adjacent to every vertex of $P$. We say that $v$ is the \emph{stem} of $G$. Then, we define a \emph{petunia} to be a graph in which every block is a subgraph of a petal graph.
\end{definition}
We note that a petal graph is an example of a \emph{fan graph}, which is constructed from a path and a coclique joined by a complete bipartite graph. We use the term \emph{petunia} rather than \emph{flower} so as not to be confused with other uses of the word \emph{flower} in combinatorics (e.g.~\cite{Rodl}).
In the following lemma, we show that petunias admit a vertex partition satisfying the conditions described in Theorem \ref{thm:tree}.
\begin{lemma}
\label{lem:petunia_partition}
If $G$ is a petunia, then $V(G)$ can be partitioned into forests $F_1, \dots, F_t$ such that the quotient graph $G / (F_1, \dots, F_t)$ is a forest and such that for each distinct pair $i,j$, $|N(F_i) \cap V(F_j)| \leq 3$.
\end{lemma}
\begin{proof}
We add edges to $G$ until every block of $G$ is a petal graph. 
Then, we will color $V(G)$ red and blue, and we will let each maximal connected monochromatic subgraph of $G$ give the vertex set of a tree $F_i$. After removing our extra edges from $G$, the subgraphs $F_1, \dots, F_t$ will make a family of forests satisfying the conditions of the lemma.

We give a general procedure for how to color the vertices of a block $H$ of $G$ with red and blue. We let $H$ consist of a stem $v$ and a path $P$. If $v$ is colored with either color, then we color the path $P$ formed by the non-stem vertices of $H$ so that $P$ alternates between red and blue. If a vertex $w \in V(P)$ is colored with either color, then extend the coloring of $w$ to the entire path $P$ formed by the non-stem vertices of $H$ so that $P$ alternates between red and blue. Then, we color $v$ with either color. Then, to color $G$, we begin by coloring a vertex of each component of $G$ with an arbitrary color, and then we extend the coloring using the procedure we have described. 
After $G$ is colored, we observe that each maximal connected monochromatic subgraph of $G$ intersects each block of $G$ either in a star or at a single vertex, and these single vertices form an independent set. Therefore, each maximal connected monochromatic subgraph of $G$ is a tree (which may span several blocks), and these trees give a partition $F_1, \dots, F_t$ of $V(G)$. 

We argue that for a pair $i,j$, the tree $F_i$ has at most three neighbors in $F_j$. If $F_i$ has at most one neighbor in $F_j$, then we are done. Otherwise, choose vertices $u,v \in V(F_i)$ so that $u$ has a neighbor $x \in V(F_j)$ and $v$ has a neighbor $y \in V(F_j)$ distinct from $x$. Since $F_i$ is a tree, there exists a path $P$ (possibly of length $0$) in $F_i$ from $u$ to $v$, and similarly, there exists a path $Q$ in $F_j$ from $x$ to $y$. Then, $V(P) \cup V(Q)$ gives the vertex set of a cycle $C$ in $G$, and $C$ must belong to a single block $H$ of $G$. It then follows that every vertex in $N(F_i) \cap V(F_j)$ belongs to the block $H$. Thus it is easy to see from our construction of $F_1, \dots, F_t$ that $F_i$ has at most three neighbors in $F_j$.

Finally, we argue that the quotient graph $G/(F_1, \dots, F_t)$ is a forest. Suppose that this quotient graph contains a cycle $C$. We assume without loss of generality that $F_1$ belongs to $C$, with neighbors $F_2$ and $F_3$. 
%It then follows that there exists a path $Q$ from any vertex $v \in F_2$ to any vertex $w \in F_3$ such that $Q$ does not intersect $F_1$. We choose two vertices $x,y \in F_1$ for which $x \sim v$ and $y \sim w$. If $Q'$ is a path in $F_1$ from $x$ to $y$, then $V(Q) \cup V(Q')$ gives the vertex set of 
Using a similar argument to that above, there exists a cycle $C'$ in $G$ that visits $F_i$, then $F_{i+1}$, and then later visits $F_{i+2}$ without once again visiting $F_i$ (with $i \in \{1,2,3\}$ and addition ``wrapping around"). Since $C'$ is two-connected, $C'$ must belong to a single block $H$ of $G$. However, again, because $H$ is partitioned by the trees $F_i$ into a star and single independent vertices, the cycle $C'$ cannot satisfy these properties.
Therefore, the quotient graph $G/(F_1, \dots, F_t)$ is a forest.
%. Therefore, in $H$, every path between two vertices outside of $S$ must pass through $S$. We observe that $S$ belongs to some $F_i$, and then the assumption that $C'$ passes from $F_{i+1}$ to $F_{i+2}$ without visiting $F_i$ is contradicted.
\end{proof}

Lemma \ref{lem:petunia_partition} shows that the hypotheses of Theorem \ref{thm:tree} hold for petunias with $r = 3$. Therefore, we can obtain an upper bound on the hat guessing number of petunias as follows.
\begin{theorem}
\label{thm:petunia}
If $G$ is a petunia, then $\HG_s(G) \leq (3s^2+3s+3)^{3(s^2+s+1)^2}$.
\end{theorem}
\begin{proof}
We partition $G$ into forests $F_i$ as described in Lemma \ref{lem:petunia_partition}. By Theorem \ref{thm:tree} (and also by a result of Bosek et al.~\cite{Jaroslaw}), each forest $F_i$ satisfies $\HG_s(F_i) \leq (s+1)s$. Then, by applying Theorem \ref{thm:tree} to each component of $G$ with $r = 3$ and $\ell = (s+1)s+1$, we obtain the result.
\end{proof}

In order to prove Theorem \ref{thm:OP_intro}, we will need one more lemma. In the following lemma, we show that using the notion of petunias, we can find a useful vertex decomposition of any outerplanar graph.
\begin{lemma}
\label{lem:petuniaOP}
If $G$ is an outerplanar graph, then $V(G)$ can be partitioned into two sets $A$ and $B$ such that $G[A]$ is a petunia, $B$ is an independent set, and each vertex of $A$ is adjacent to at most three vertices of $B$.
\end{lemma}
\begin{proof}
If $|V(G)| \leq 1$, then we let $A= V(G)$, and we are done. Otherwise, we assume $|V(G)| \geq 2$.
Since the class of petunias is closed under taking subgraphs, we may add edges to $G$ until $G$ is a maximal outerplanar graph (or equivalently an outerplanar triangulation whenever $|V(G)| \geq 3$), and doing so will not make our task any easier.

We prove the following stronger claim:
\begin{quote}
Let $G$ be a maximal outerplanar graph on at least two vertices, and let $uv$ be an edge of $G$ oriented from $u$ to $v$. Then $V(G)$ can be partitioned into two sets $A$ and $B$ such that the following hold:
\begin{enumerate}
\item \label{lab:uv} $G[A]$ is a petunia containing $uv$;
\item $B$ is an independent set;
\item \label{lab:u} $u$ is not adjacent to any vertex of $B$;
\item \label{lab:v} $v$ is adjacent to at most two vertices of $B$;
\item \label{lab:others} Every other vertex of $G$ is adjacent to at most three vertices of $B$.
\end{enumerate}
\end{quote}
We prove the statement by induction on $|V(G)|$. If $|V(G)| =2$, then again letting $A = V(G)$ satisfies the claim. Now, suppose that $|V(G)| \geq 3$, and let $uv \in E(G)$. We fix an outerplanar drawing of $G$.
Since $G$ is an outerplanar triangulation, it follows that $G$ is $2$-connected, and hence $G$ contains a Hamiltonian cycle $C$ separating the outer face of $G$ from all other faces. 

We begin to construct our sets $A$ and $B$ as follows.
We first add all vertices in $N[u]$ to $A$. Note that $N[u]$ induces a petunia in $G$. Now, let $t = \deg(u)$, and write $w_1$ and $w_t$ for the neighbors of $u$ via $C$, and assume without loss of generality that the neighbors of $u$, in clockwise order, are $w_1, w_2, \dots, w_t$. We observe that any given component $K$ of $G \setminus N[u]$ is separated from the rest of $G$ by two vertices of the form $w_i, w_{i+1}$, for some $1 \leq i \leq t-1$. We show an example of such a component $K$ in Figure \ref{fig:induction}.

Now, if $N[u]$ contains all vertices of $G$, then it is easy to check that we are done. Otherwise, let $K$ be some component of $G \setminus N[u]$, and let $K$ be separated from the rest of $G$ by $w_i, w_{i+1}$. Since $G$ is an outerplanar triangulation, $uw_i$ and $uw_{i+1}$ lie on some triangle $T$ in the interior of $C$, so $w_i$ and $w_{i+1}$ are adjacent. Furthermore, since $w_i$ and $w_{i+1}$ separate some component $K$ from the rest of $G$, it must follow that the edge $w_iw_{i+1}$ does not belong to $C$, and hence $T$ shares the edge $w_i w_{i+1}$ with a second triangle $T'$ in the interior of $C$. The triangle $T'$ includes the vertices $w_i$ and $w_{i+1}$, along with a third vertex $x_i$. We add $x_i$ to $B$, and we add all neighbors of $x_i$ to $A$.

\begin{figure}
  \begin{center}
\begin{tikzpicture}
  [scale=1.5,auto=left,every node/.style={circle,minimum size = 6pt,inner sep=0pt}, decoration={markings, 
    mark= between positions 0.53 and 0.53 step 0.1mm 
        with {\arrow{latex}}}]

\node(z) at (5,4.45) [draw = white, fill = white] {$u$};
\node(z) at (5,3) [draw = white, fill = white] {$T$};
\node(z) at (7,1) [draw = white, fill = white] {$T'$};
\node(z) at (0,-0.25) [draw = white, fill = white] {$w_{i+1} = y_s$};
\node(z) at (2,-0.25) [draw = white, fill = white] {$y_{s-1}$};
\node(z) at (4,-0.25) [draw = white, fill = white] {$\cdots$};
\node(z) at (6,-0.25) [draw = white, fill = white] {$y_{j+1}$};
\node(z) at (7,-0.25) [draw = white, fill = white] {$x_i$};
\node(z) at (8,-0.25) [draw = white, fill = white] {$y_{j}$};
\node(z) at (9,-0.25) [draw = white, fill = white] {$\cdots$};
\node(z) at (10,-0.25) [draw = white, fill = white] {$w_i = y_1$};

\node(u) at (5,4.2) [draw = black, fill = black] {};
\node (wi) at (10,0) [draw = black, fill = black] {};
\node (wi1) at (0,0) [draw = black, fill = black] {};
\node (xi) at (7,0) [draw = black, fill = white] {};
\node (y1) at (2,0) [draw = black, fill = black] {};
\node (y2) at (4,0) [draw = black, fill = black] {};
\node (y3) at (6,0) [draw = black, fill = black] {};
\node (y4) at (8, 0) [draw = black, fill = black] {};
\node (y5) at (9, 0) [draw = black, fill = black] {};

    \draw
      (wi1) to[out=35,in=145,looseness=1.3] (xi) {};
    \draw
      (xi) to[out=40,in=140,looseness=1.5] (wi) {};
    \draw
      (xi) to[out=38,in=142,looseness=1.2] (y5) {};
    \draw
      (wi1) to[out=40,in=140,looseness=1.15] (wi) {};
    \draw[name path = A]
      (wi1) to[out=30,in=150,looseness=1.1] (y1) {};
	\draw[name path =A1]
	(wi1) to[out=0,in=180,looseness=0] (y1) {};
    \draw[name path = B]
      (y1) to[out=30,in=150,looseness=1.1] (y2) {};
    \draw[name path = B1]
      (y1) to[out=0,in=180,looseness=0] (y2) {};
    \draw[name path = C]
      (y2) to[out=30,in=150,looseness=1.1] (y3) {};
    \draw[name path = C1]
      (y2) to[out=0,in=180,looseness=0] (y3) {};
    \draw
      (y2) to[out=32,in=148,looseness=1.2] (xi) {};
    \draw
      (y1) to[out=32,in=148,looseness=1.3] (xi) {};
\foreach \from/\to in {u/wi,wi1/u,y1/y2,y2/y3,y3/xi,xi/y4,y4/wi}
    \draw (\from) -- (\to);

\draw[name path = D]
      (y4) to[out=40,in=140,looseness=1.1] (y5) {};
    \draw[name path = D1]
      (y4) to[out=0,in=180,looseness=0] (y5) {};
\draw[name path = E]
      (y5) to[out=30,in=150,looseness=1.1] (wi) {};
    \draw[name path = E1]
      (y5) to[out=0,in=180,looseness=0] (wi) {};

\tikzfillbetween[of=A and A1]{gray!40};
\tikzfillbetween[of=B and B1]{gray!40};
\tikzfillbetween[of=C and C1]{gray!40};
\tikzfillbetween[of=D and D1]{gray!40};
\tikzfillbetween[of=E and E1]{gray!40};

    \draw[postaction = {decorate}]
      (wi1) to[out=30,in=150,looseness=1.1] (y1) {};
	\draw
	(wi1) to[out=0,in=180,looseness=0] (y1) {};
    \draw[postaction = {decorate}]
      (y1) to[out=30,in=150,looseness=1.1] (y2) {};
    \draw
      (y1) to[out=0,in=180,looseness=0] (y2) {};
    \draw[postaction = {decorate}]
      (y2) to[out=30,in=150,looseness=1.1] (y3) {};
    \draw
      (y2) to[out=0,in=180,looseness=0] (y3) {};

    \draw[postaction = {decorate}]
      (wi) to[out=150,in=30,looseness=1.1] (y5) {};
	\draw
	(wi) to[out=0,in=180,looseness=0] (y5) {};
    \draw[postaction = {decorate}]
      (y5) to[out=140,in=40,looseness=1.1] (y4) {};
	\draw
	(y5) to[out=0,in=180,looseness=0] (y4) {};

\node (z) at (0.5,0) [draw = black,fill=gray!25] {};
\node (z) at (1,0) [draw = black,fill=gray!25] {};
\node (z) at (1.5,0) [draw = black,fill=gray!25] {};

\node (z) at (2.5,0) [draw = black,fill=gray!25] {};
\node (z) at (3,0) [draw = black,fill=gray!25] {};
\node (z) at (3.5,0) [draw = black,fill=gray!25] {};

\node (z) at (4.5,0) [draw = black,fill=gray!25] {};
\node (z) at (5,0) [draw = black,fill=gray!25] {};
\node (z) at (5.5,0) [draw = black,fill=gray!25] {};

\node (z) at (8.5,0) [draw = black,fill=gray!25] {};
\node (z) at (9.5,0) [draw = black,fill=gray!25] {};

\end{tikzpicture}
\end{center}
\caption{The figure shows part of the outerplanar graph $G$ from the proof of Lemma \ref{lem:petuniaOP}. The black vertices in the figure belong to the set $A$, and the white vertex belongs to the set $B$. Each shaded region represents some outerplanar subgraph of $G$. We may partition $V(G)$ as described in Lemma \ref{lem:petuniaOP} by applying the induction hypothesis to each of the outerplanar graphs shown as a shaded region in the figure.}
\label{fig:induction}
\end{figure}
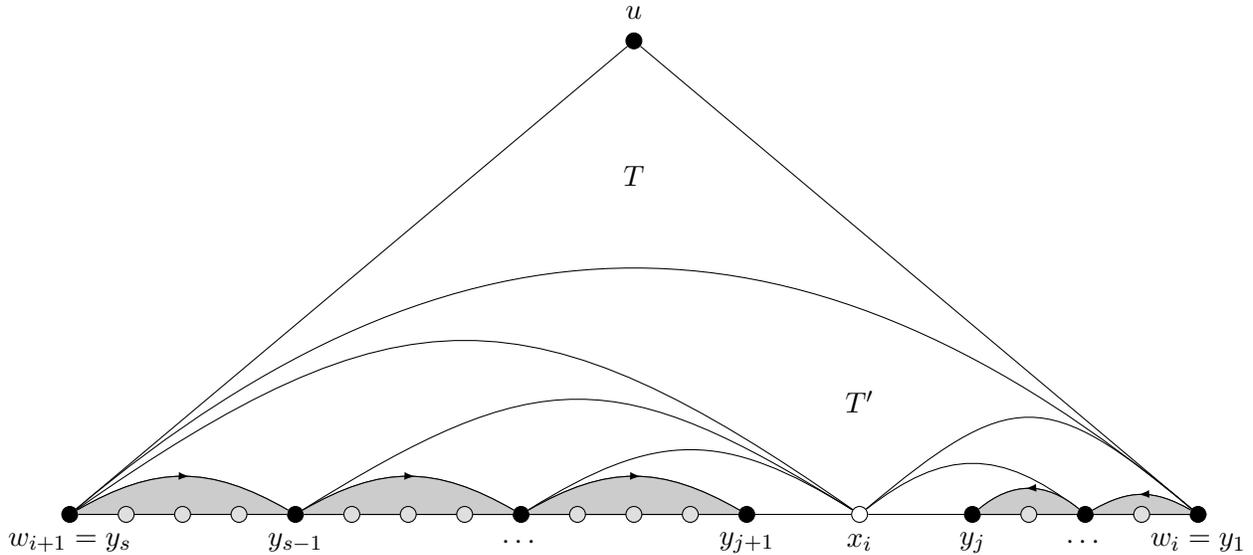

Now, let the neighbors of $x_i$, in clockwise order, be $y_1, \dots, y_s$, with $y_1 = w_i$ and $y_s = w_{i+1}$. Let $y_j$ and $y_{j+1}$ be neighbors of $x_i$ along $C$. Since $G$ is an outerplanar triangulation, it follows that $G$ contains the edge $y_{\alpha}y_{\alpha+1}$ for $1 \leq \alpha \leq j-1$, as well as for $j + 1 \leq \alpha \leq s-1$. Now, for $1 \leq \alpha \leq j-1$, we orient the edge $y_{\alpha} y_{\alpha + 1}$ from $y_{\alpha}$ to $y_{\alpha + 1}$. Then, for $j + 1 \leq \alpha \leq s-1$, we orient the edge $y_{\alpha} y_{\alpha + 1}$ from $y_{\alpha + 1}$ to $y_{\alpha}$. Finally, for each value $\alpha \in [s-1] \setminus j$, we apply the induction hypothesis to the outerplanar subgraph of $G$ that is either $2$-connected or isomorphic to $K_2$ whose outer facial walk is given by the edge $y_{\alpha} y_{\alpha + 1}$ along with the path from $y_{\alpha}$ to $y_{\alpha + 1}$ along $C$. By induction, all vertices of $G$ are partitioned into the sets $A$ and $B$.

We claim that all criteria of the induction statement are satisfied. First, we must show that $G[A]$ is a petunia containing $uv$. By construction, $A$ contains $u$ and $v$. Also, clearly $N[u]$ is a petunia, as $G$ is outerplanar. Furthermore, each vertex $y_{\alpha}$ described in the process above must be a cut-vertex in $G[A]$, so $G[A]$ remains a petunia even after adding vertices using the induction hypothesis. Second, clearly no pair $x_i, x_j$ is adjacent, so $B$ is initially an independent set. Furthermore, as all neighbors of each $x_i$ are added to $A$, $B$ remains an independent set even after applying the induction hypothesis. Third, by construction, $u$ is not adjacent to any vertex of $B$. Fourth, if $v = w_i$, then as $G$ is outerplanar, $v$ is initially adjacent to at most two vertices of $B$, namely $x_i$ and $x_{i-1}$. Furthermore, each vertex $w_i$ is the tail of an arc in all outerplanar graphs containing $w_i$ for which the induction hypothesis is applied, and thus $w_i$ does not gain any neighbors in $B$ after applying induction because of criterion (\ref{lab:u}). 

Finally, we argue that each vertex $z \in A$ is adjacent to at most three vertices of $B$. If $z$ is of the form $w_i$ or $u$, then (\ref{lab:others}) holds for $z$. Otherwise, if $z$ is of the form $y_{\alpha}$ in the process described above, then $z$ belongs to at most two outerplanar graphs $H, H'$ for which the induction hypothesis is called, and $z$ belongs to an arc of both $H$ and $H'$. Furthermore, $z$ is the head of at most one of these arcs. Therefore, by criteria (\ref{lab:u}) and (\ref{lab:v}), $z$ gains at most two neighbors in $B$ during induction. As we have assumed that $z$ is not of the form $u$ or $w_i$, it follows that $z$ is adjacent to at most one vertex of the form $x_i$. Therefore, $z$ has at most three neighbors in $B$. Finally, if none of the above holds, then $z$ belongs to an outerplanar graph $H$ on which induction is applied, and $z$ is separated from all vertices $x_i$ by some edge $y_{\alpha} y_{\alpha + 1}$. Therefore, by criterion (\ref{lab:others}), $z$ has at most three neighbors in $B$. Thus the induction statement holds, and the proof is complete.
\end{proof}

With Lemma \ref{lem:petuniaOP} proven, we are ready to apply Lemma \ref{lem:partition} and Theorem \ref{thm:petunia} to obtain an upper bound for the hat guessing number of outerplanar graphs. Letting $s = 1$ in the following theorem immediately implies Theorem \ref{thm:OP_intro}.

\begin{theorem}
\label{thm:OP}
If $G$ is an outerplanar graph, then 
\[ \HG_s(G) \leq  (3(s+1)^6+3(s+1)^3 + 3 )^{3((s+1)^6+(s+1)^3+1)^2}. \]
\end{theorem}
\begin{proof}
We partition $G$ using Lemma \ref{lem:petuniaOP} so that $G[A]$ is a petunia, $B$ is an independent set, and each vertex of $A$ has at most three neighbors in $B$. Since $\HG_s(G[B]) = s$, we follow Lemma \ref{lem:partition} and set $d =3$ and $s' = (s+1)^3$. Then, according to Lemma \ref{lem:partition} and Theorem \ref{thm:petunia}, 
\[ \HG_s(G) \leq \HG_{(s+1)^3}(G[A])  \leq (3(s+1)^6+3(s+1)^3 + 3 )^{3((s+1)^6+(s+1)^3+1)^2}.\]
\end{proof}

\section{Layered planar graphs}
\label{sec:planar}
In this section, we compute an upper bound for the hat guessing number of layered planar graphs. Formally, we define layered planar graphs as follows.
\begin{definition}
Consider a planar graph $H$ obtained from the following process. We begin with a $2$-connected outerplanar graph $G_1$ embedded in the plane so that the unbounded face is incident to all vertices of $G_1$. Then, we choose some integer $\tau \geq 1$, and for each $2 \leq i \leq \tau$, we draw a $2$-connected outerplanar $G_i$ inside some interior face of $G_{i-1}$ so that in the drawing of $G_i$, the unbounded face contains all vertices of $G_{i}$. Then, we add some set of edges between $G_{i-1}$ and $G_i$ in such a way that does not introduce a crossing. If $G$ is a subgraph of a graph $H$ constructed in this way, then we say that $G$ is a \emph{layered planar graph.}
\end{definition}

We obtain the following upper bound for the hat guessing number of layered planar graphs. Letting $s = 1$ in the following theorem immediately implies Theorem \ref{thm:planar_intro}.

\begin{theorem}
\label{thm:planar}
If $G$ is a layered planar graph, then $\log_2 \log_2 \log_2 \log_2 \HG_s(G) <  2^{149}s^{35} $.
\end{theorem}
\begin{proof}
We fix a drawing of $G$ in the plane. We partition the vertices of $G$ into \emph{levels} $L_i$ as follows. First, we let $L_1$ denote the set of vertices on the outer face of $G$. Then, for $i \geq 1$, we let $L_{i+1}$ denote the set of vertices on the outer face of $G \setminus (L_1 \cup \dots \cup L_i)$. 
Since $G$ is a layered planar graph, we may assume that $G_i$ is a $2$-connected outerplanar graph for each level $L_i$, and that every edge of $G$ either has both endpoints in some $L_i$ or has one endpoint in some $L_i$ and the other endpoint in $L_{i+1}$. If a vertex $v \in L_i$ has a neighbor $u$, then we say that $u$ is a \emph{parent} of $v$ if $u \in L_{i-1}$, a \emph{sibling} of $v$ if $u \in L_i$, and a \emph{child} of $v$ if $u \in L_{i+1}$.

Now, we will begin to partition the vertices of $G$ into color classes. Initially, we let every vertex of $G$ be colored \texttt{blank}. Then, for each vertex $v$ in each level $L_i$, let $u_1, \dots, u_t$ be the parents of $v$ in clockwise order. 
We let $K_v$ denote the (possibly empty) subgraph of $G_{i-1}$ that is separated from the rest of $G_{i-1}$ by $u_1$ and $u_t$ and that can be reached by travelling from $v$ to $u_1$ and then turning right. In other words, $K_v$ is on the ``clockwise side" of the arc $v u_1$ and the ``counterclockwise side" of $v u_t$, and if $K_v$ is nonempty, it contains the vertices $u_2, \dots, u_{t-1}$. We color every vertex of $K_v$ green, as shown in Figure \ref{fig:green}. Observe that by planarity, since $G_i$ and $G_{i-1}$ are $2$-connected, each vertex of $K_v$ can only have $v$ as a child.

\begin{figure}
\begin{center}
\begin{tikzpicture}
[scale=1.2,auto=left,every node/.style={circle,fill=gray!30,minimum size = 6pt,inner sep=0pt}]

\node(z) at (0,-0.35) [draw = white, fill = white] {$v$};
\node(z) at (-3,1.35) [draw = white, fill = white] {$u_1$};
\node(z) at (3,1.35) [draw = white, fill = white] {$u_t$};

\node(z) at (-4.5,0) [draw = white, fill = white] {$L_i$};
\node(z) at (-4.5,1) [draw = white, fill = white] {$L_{i-1}$};

\draw[thick] (-4,0) -- (4,0);
\draw[thick] (-4,1) -- (4,1);
\node(v) at (0,0) [draw = black] {};
\node(u1) at (-3,1) [draw = black] {};
\node(u2) at (-2,1) [draw = black, fill = ForestGreen] {};
\node(u3) at (-1,1) [draw = black, fill = ForestGreen] {};
\node(u4) at (-0,1) [draw = black, fill = ForestGreen] {};
\node(u5) at (1,1) [draw = black, fill = ForestGreen ]{};
\node(u6) at (2,1) [draw = black, fill = ForestGreen] {};
\node(u7) at (3,1) [draw = black] {};
\draw (u3) to [out=-30,in=-160] (u6);
\draw (u3) to [out=-20,in=-170] (u5);

\foreach \from/\to in {v/u1,v/u3,v/u6,v/u7,v/u2}
    \draw (\from) -- (\to);

\end{tikzpicture}
\end{center}
\caption{The figure shows a vertex $v \in L_i$ with $u_1$ as its counterclockwise-most parent and with $u_t$ as its clockwise-most parent. We use green to color all vertices in $L_{i-1}$ on the clockwise side of $v u_1$ and on the counterclockwise side of $v u_t$.}
\label{fig:green}
\end{figure}
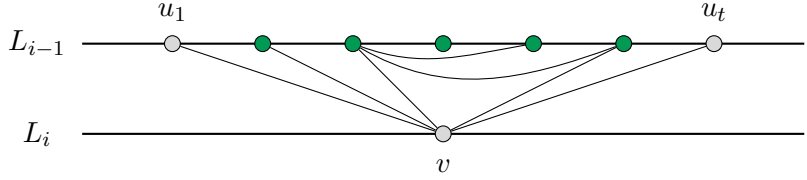

%We color the vertices $u_2, \dots, u_{t-1}$ green. Note that a vertex with two or fewer parents will not have any of its parents colored green. 

%{\PB We will need that green graph is outerplanar and that every green vertex has at most five non-green neighbors.}

Next, for every vertex $v \in V(G)$, if $v$ has at least three children, then we color $v$ red. Since each green vertex has only one child, no vertex will be colored both green and red. Then, if a red vertex $v$ has at least one red child, then we use pink to recolor the clockwise-most and counterclockwise-most red child of $v$, as shown in Figure \ref{fig:red}. Finally, we use blue to color all remaining \texttt{blank} vertices in a level $L_i$ with $i$ even, and we use indigo to color all remaining \texttt{blank} vertices in a level $L_i$ with $i$ odd.

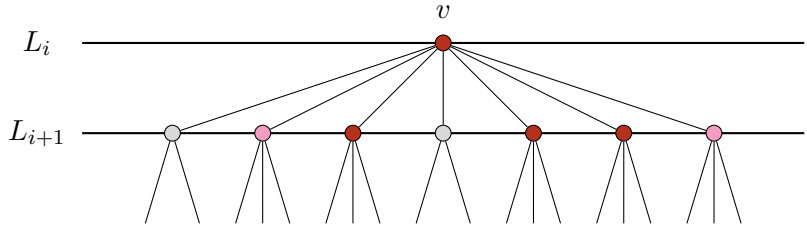
\begin{figure}
\begin{center}
\begin{tikzpicture}
[scale=1.2,auto=left,every node/.style={circle,fill=gray!30,minimum size = 6pt,inner sep=0pt}]

\node(z) at (0,2.35) [draw = white, fill = white] {$v$};
%\node(z) at (-3,1.35) [draw = white, fill = white] {$u_1$};
%\node(z) at (3,1.35) [draw = white, fill = white] {$u_t$};

\node(z) at (-4.5,2) [draw = white, fill = white] {$L_i$};
\node(z) at (-4.5,1) [draw = white, fill = white] {$L_{i+1}$};

\draw[thick] (-4,2) -- (4,2);
\draw[thick] (-4,1) -- (4,1);
\node(v) at (0,2) [draw = black, fill = BrickRed] {};
\node(u1) at (-3,1) [draw = black] {};
\draw (u1) -- (-3.3,0);
\draw (u1) -- (-2.7,0);
\node(u2) at (-2,1) [draw = black, fill =  Lavender] {};
\draw (u2) -- (-2.3,0);
\draw (u2) -- (-2,0);
\draw (u2) -- (-1.7,0);
\node(u3) at (-1,1) [draw = black, fill =  BrickRed] {};
\draw (u3) -- (-1.3,0);
\draw (u3) -- (-1,0);
\draw (u3) -- (-0.7,0);
\node(u4) at (-0,1) [draw = black] {};
\draw (u4) -- (-0.3,0);
\draw (u4) -- (0.3,0);
\node(u5) at (1,1) [draw = black, fill = BrickRed]{};
\draw (u5) -- (1.3,0);
\draw (u5) -- (1,0);
\draw (u5) -- (0.7,0);
\node(u6) at (2,1) [draw = black, fill =  BrickRed] {};
\draw (u6) -- (2.3,0);
\draw (u6) -- (2,0);
\draw (u6) -- (1.7,0);
\node(u7) at (3,1) [draw = black, fill = Lavender] {};
\draw (u7) -- (3.3,0);
\draw (u7) -- (3,0);
\draw (u7) -- (2.7,0);
%\draw (u3) to [out=-30,in=-160] (u6);
%\draw (u3) to [out=-20,in=-170] (u5);

\foreach \from/\to in {v/u1,v/u3,v/u4,v/u5,v/u6,v/u7,v/u2}
    \draw (\from) -- (\to);

\end{tikzpicture}
\end{center}
\caption{The figure shows a red vertex $v \in L_i$ with several red children. We use pink to recolor the clockwise-most and counterclockwise-most red children of $v$.}
\label{fig:red}
\end{figure}

%{\PB We will need that the red graph can be partitioned into a petunia and a nice IS. We also need that \texttt{blank} doesn't have many red neighbors. Not a lot of red parents, and not a lot of red kids. Also not a lot of red siblings by virtue of being red. We have green->(blue->red).}

Now, we make a series of claims about our coloring of $G$. The ultimate goal of these claims will be to show that we can apply Lemma \ref{lem:gen_partition} to obtain an upper bound for $\HG_s(G)$.

\begin{claim}
The green vertices of $G$ induce an outerplanar graph, and every green vertex has at most five neighbors in a color other than green.
\end{claim}
\begin{proof}
We first show that each green vertex has at most five non-green neighbors. Let $v$ be a green vertex. Since all parents of $v$ except for the clockwise-most and counterclockwise-most parent are colored green, $v$ has at most two non-green parents. If $v \in L_i$, then by construction, $v$ belongs to a connected green subgraph in $G_i$ that is separated from the rest of $G_i$ by two non-green vertices. Therefore, $v$ also has at most two non-green siblings. Finally we have observed previously that $v$ has at most one child. Therefore, $v$ has at most five non-green neighbors.

Now, we show that the green vertices of $G$ induce an outerplanar graph. Since each $G_i$ is an outerplanar graph, clearly the green vertices in any single level $L_i$ induce an outerplanar graph. Furthermore, by construction, for any two distinct green vertices $u,v \in L_i$, no green parent of $u$ is equal to or adjacent to a green parent of $v$, and hence no ancestor of $u$ in the green induced subgraph is equal to or adjacent to an ancestor of $v$ in the green induced subgraph. Therefore, the green induced subgraph of $G$ is a graph in which each block is an outerplanar graph contained in some level $L_i$, and hence the green induced subgraph of $G$ is outerplanar.
\end{proof}
\begin{claim}
\label{claim:blue}
The blue vertices of $G$ induce an outerplanar graph, and every blue vertex has at most six neighbors in indigo, red, or pink.
\end{claim}
\begin{proof}
Clearly the blue vertices induce an outerplanar graph, as they induce a subgraph of the disjoint union of the outerplanar graphs $G_i$ for even values $i$. 

Now, let $v \in L_i$ be a blue vertex. As $v$ is not red or pink, $v$ has at most two indigo, red, or pink children. As all but at most two parents of $v$ are colored green, $v$ has at most two indigo, red, or pink parents. We also observe that by construction, $v$ has no indigo sibling. 

Finally, we argue that $v$ has at most two red or pink siblings. If $G_{i+1}$ is empty, then clearly $v$ does not have a red or pink sibling, so we assume that $G_{i+1}$ is a $2$-connected outerplanar graph.
Since $G_i$ is an outerplanar graph, $G_i$ has a Hamiltonian cycle $C_i$ such that $E(C_i)$ and the interior of $E(C_i)$ contain all edges of $G_i$. Suppose that $v$ has at least three red or pink siblings. Assume that when starting outside $C_i$ and then visiting the edges incident to $v$ in clockwise order, we visit edges incident to three red or pink siblings $u_1, u_2, u_3$, in order. Observe that since $G_i$ is outerplanar, the edge $v u_2$ separates the interior of $C_i$ into two regions, one containing $u_1$ and one containing $u_3$. However, since $G_{i+1}$ is connected, the edge $v u_2$ separates one of $u_1, u_3$ from $G_{i+1}$, contradicting the assumption that both of these vertices have neighbors in $G_{i+1}$. Therefore, $v$ has at most two red or pink siblings.
\end{proof}
\begin{claim}
The indigo vertices of $G$ induce an outerplanar graph, and every indigo vertex has at most six neighbors in red or pink.
\end{claim}
\begin{proof}
The proof is similar to that of Claim \ref{claim:blue}.
\end{proof}
\begin{claim}
\label{claim:red}
The red vertices in $G$ induce a petunia, and every red vertex has at most six pink neighbors.
\end{claim}
\begin{proof}
In our proof, we will often use the fact that for each value $1 \leq i \leq \tau $, the subgraph of $G_i$ induced by those vertices with at least one child is a subgraph of a cycle. This fact follows easily from planarity.

To show that the red vertices in $G$ induce a petunia, we first claim that if two distinct red vertices $u,v$ in a common level $L_i$
have respective red children $u'$ and $v'$, then $u'$ and $v'$ are not joined by a path of red vertices in $L_{i+1}$. 
Indeed, $u$ has two pink children $w, w' \in L_{i+1}$ with respective children $x,x' \in L_{i+2}$, and the edges $wx$ and $w'x'$ along with $G_{i+2}$ separate $u'$ from $v'$. Thus, the claim holds, and by a similar argument, $u'$ and $v'$ cannot be equal.

Next, consider the subgraph $G'$ of $G$ induced by the red vertices of $G$, and consider a $2$-connected subgraph $H$ of $G'$. If all vertices of $H$ belong to a single level $L_i$, then $H$ is a subgraph of a cycle and hence a petal graph. Otherwise, by the previous observation, for each vertex $v \in V(G')$, $v$ separates the descendants of $v$ in $G'$ from all other vertices in $G'$. Therefore, if two vertices of $H$ belong to a level $L_i$, then as $H$ is $2$-connected, no vertex of $H$ can belong to $L_{i+1}$. Therefore, it must follow that $H$ contains exactly one vertex in some level $L_i$ and that all other vertices of $H$ belong to $L_{i+1}$. As seen before, the red and pink vertices in $L_{i+1}$ induce a subgraph of a cycle, and hence the red vertices of $H$ in $L_{i+1}$ induce a subgraph of a path. Therefore, $H$ is a petal graph, and $G'$ is a petunia.

Now, let $v$ be a red vertex. As argued before, $v$ has at most two non-green parents. By the same argument used in Claim \ref{claim:blue}, $v$ has at most two pink siblings. Finally, if $v \in G_i$, then since $G_{i+1}$ is $2$-connected and outerplanar, $v$ has at most two pink children.
\end{proof}
\begin{claim}
The pink vertices in $G$ induce a graph of maximum degree $6$.
\end{claim}
\begin{proof}
The proof is similar to that of Claim \ref{claim:red}.
\end{proof}

Now, we are ready to apply Lemma \ref{lem:gen_partition} and obtain an upper bound for the hat guessing number of planar graphs. Following Lemma \ref{lem:gen_partition}, we let $V_1$ denote the green vertices of $G$, $V_2$ the blue vertices, $V_3$ the indigo vertices, $V_4$ the red vertices, and $V_5$ the pink vertices. Then, we define the following values:
\begin{eqnarray*}
\ell_5 & = & 20s \\
s_4 & =&s \ell_5^6 = (20)^6 s^7 \\
\ell_4 & = & (3s_4^2 + 3s_4 + 3)^{3(s_4^2+s_4+1)^2} + 1< 2^{2^{138}s^{30}} \\
s_3 & = &s \ell_4^6  < 2^{2^{141}s^{35}} \\
\ell_3 & = & (3(s_3+1)^6+3(s_3+1)^3 + 3 )^{3((s_3+1)^6+(s_3+1)^3+1)^2}+1< 2^{2^{2^{145}s^{35}}} \\
s_2 & = &s \ell_3^6  <2^{2^{2^{146}s^{35}}} \\
\ell_2 & = &   (3(s_2+1)^6+3(s_2+1)^3 + 3 )^{3((s_2+1)^6+(s_2+1)^3+1)^2} +1< 2^{2^{2^{2^{147}s^{35}}}} \\
s_1 & = &s \ell_2^5 < 2^{2^{2^{2^{148}s^{35}}}} \\
\ell_1 & = &   (3(s_1+1)^6+3(s_1+1)^3 + 3 )^{3((s_1+1)^6+(s_1+1)^3+1)^2}+1< 2^{2^{2^{2^{2^{149}s^{35}}}}} \\
\end{eqnarray*}
We verify these estimates in the appendix. Since the upper bound given by this method is probably too large, we make no real effort to optimize our estimates.

Now, we must check that all of the hypotheses of Lemma \ref{lem:gen_partition} hold. It is easy to check from our claims that we have given appropriate definitions to each value $s_i$. (In fact, we may overestimate the values of $s_1, s_2, s_3$, but this is fine.) Then, we show that each $\ell_i$ is large enough as follows.

As the pink vertices induce a subgraph of maximum degree at most $6$, it follows from Lemma \ref{lem:LLL} that $\HG_s(G[V_5]) < 20s$. As $G[V_4]$ is a petunia, $\HG_{s_4}(G[V_4]) < \ell_4$ by Theorem \ref{thm:petunia}. Finally, as $G[V_i]$ is outerplanar for $i \in \{1,2,3\}$, $\HG_{s_i}(G[V_i]) < \ell_i$ for $i \in \{1,2,3\}$ by Theorem \ref{thm:OP}. Therefore, $\HG_s(G) < \ell_1$, and the proof is complete.
\end{proof}

\section{Graphs of bounded genus}
\label{sec:genus}
While it is still unknown whether the hat guessing number of planar graphs is bounded, it is straightforward to show that if $\HG_s(H)$ is bounded for every planar graph $H$, then $\HG_s(G)$ is also bounded for every graph $G$ of bounded genus. We will use the following lemma of Mohar and Thomassen, which follows from first principles of algebraic topology. For a graph $G$ embedded on a surface $S$, we say that a cycle $C$ in $G$ is \emph{separating} if $S \setminus C$ has at least two connected components, or equivalently, if $C$ is zero-homologous.

\begin{lemma}[\cite{MoharThomassen}]
\label{lem:cycle}
Let $G$ be a graph embedded on a surface, and let $x$ and $y$ be two distinct vertices of $G$. If $P_1$, $P_2$, and $P_3$ are distinct internally disjoint paths with endpoints $x$ and $y$, and if $P_1 \cup P_2$ and $P_1 \cup P_3$ are both separating cycles, then $P_2 \cup P_3$ is also a separating cycle.
\end{lemma}
This lemma will allow us to use a straightforward inductive argument to prove the following theorem.
\begin{theorem}
\label{thm:genus}
If $f$ is a function such that every planar graph $H$ satisfies $\HG_s(H) < f(s)$, then every graph $G$ of genus $g$ satisfies $\HG_s(G) < f(3^{2(6^g-1)}s^{6^g})$.
\end{theorem}
\begin{proof}
If $g = 0$, then $G$ is planar, and there is nothing to prove. Otherwise, we assume that $g \geq 1$. Since $G$ has no planar embedding, we must be able to find some non-separating cycle $C$ in $G$. We choose $C$ to be a shortest non-separating cycle.

Now, we claim that every vertex $v \in V(G) \setminus V(C)$ has at most five neighbors in $C$. If $|V(C)| \leq 5$, then this claim clearly holds. Otherwise, suppose that $C$ is of length at least six. If $v$ has at least six neighbors in $C$, then we must be able to choose two neighbors $x,y \in V(C)$ of $v$ that are at a distance of at least three along $C$. We define the path $P_1 = (x, v, y)$. We also define the path $P_2$ to be a shortest path from $x$ to $y$ along $C$, and we define $P_3$ to be the path with edge set $E(C) \setminus E(P_2)$. We observe that the length of $P_1 \cup P_2$ is at most $\frac{1}{2}|V(C)| + 2 < |V(C)|$, and since $P_3$ has at most $|V(C)| - 3$ edges, the path $P_1 \cup P_3$ has length at most $|V(C)| - 1$. Since $C$ is a shortest nonseparating cycle, it follows that $P_1 \cup P_2$ and $P_1 \cup P_3$ are both separating cycles. However, then Lemma \ref{lem:cycle} implies that $P_2 \cup P_3 = C$ is a separating cycle, a contradiction. Thus our claim holds.

Now, we apply Lemma \ref{lem:partition}. We let $A = V(G) \setminus V(C)$ and let $B = V(C)$. Observe that since $C$ is a nonseparating cycle, $G[A]$ has genus at most $g-1$. Following Lemma \ref{lem:partition}, we set $d = 5$. Furthermore, since $G[B]$ has maximum degree $2$, it follows from Lemma \ref{lem:LLL} that $\HG_s(G[B]) < 3es < 9s$, so we let $s' = (9s)^5 s = 3^{10}s^6$. Then, according to Lemma \ref{lem:partition}, $\HG_s(G) \leq \HG_{s'}(G[A])$. By the induction hypothesis,
\begin{eqnarray*}
HG_{s'}(G[A]) &<& f(3^{2(6^{g-1}-1)}{s'}^{6^{g-1}})  \\
 & = & f(3^{2\cdot 6^{g-1}-2}(3^{10}s^6)^{6^{g-1}}) \\
 &=& f(3^{12 \cdot 6^{g-1} - 2}s^{6^g})  \\
& = & f(3^{2 (6^{g} -1 ) }s^{6^g}).
\end{eqnarray*}
This completes the proof.
\end{proof}

\section{Conclusion}
While we are not able to prove an upper bound for the hat guessing number of planar graphs, one principle that is clear from our method is that by bounding $\HG_s(G)$ for graphs $G$ in some small graph class, it is often possible to use such a bound along with some vertex partitioning method to obtain an upper bound for the hat guessing number of a larger graph class. Indeed, in order to obtain our upper bound for the hat guessing number of layered planar graphs, we started with an upper bound on $\HG_s(F)$ for forests $F$, and then we extended this result to an upper bound for petunias, and then we extended this result to an upper bound for outerplanar graphs and finally for layered planar graphs. We hope that our upper bound in Theorem \ref{thm:planar}, along with some clever observations, will be enough to bound the hat guessing number of all planar graphs and give an affirmative answer to Question \ref{ques:main}.

An obvious question remains that we have not yet addressed: Do our upper bounds in Theorems \ref{thm:OP_intro} and \ref{thm:planar_intro} really need to be so enormous? One of the main reasons for the sheer size of our upper bounds comes from our application of Erd\H{o}s's edge density theorem in Lemma \ref{lem:Erdos1}, and while we have shown that edge density arguments can be useful for bounding a graph's hat guessing number, we have by no means shown that they are necessary. However, He, Ido, and Przybocki \cite{HeWindmill} have shown that complete multipartite graphs are closely related to hat guessing on complete bipartite graphs and on cliques joined at a cut-vertex by working in an equivalent setting that considers ``prism" subsets of $\mathbb N^r$, and a similar idea appears in \cite{Jaroslaw}. Therefore, we suspect that the hat guessing problem is closely related to classical Tur\'an-type edge density problems, and while we do not believe that our upper bounds are best possible, we believe that even for outerplanar graphs, the correct upper bound for hat guessing number should be quite large. Therefore, we ask the following question, for which we suspect that the correct answer is at least $1000$.

\begin{question}
What is the largest value $\HG(G)$ that can be achieved by an outerplanar graph $G$?
\end{question}

\section{Acknowledgment}
I am grateful to Bojan Mohar for reading an earlier draft of this paper and for suggesting Lemma \ref{lem:cycle} as a way to improve a previous version of Theorem \ref{thm:genus}. I am also grateful to an anonymous referee for giving valuable feedback that improved the presentation of this paper.

\bibliographystyle{plain}
\bibliography{HatBib}

\begin{thebibliography}{10}

\bibitem{Alon}
Noga Alon, Omri Ben-Eliezer, Chong Shangguan, and Itzhak Tamo.
\newblock The hat guessing number of graphs.
\newblock {\em J. Combin. Theory Ser. B}, 144:119--149, 2020.

\bibitem{Rodl}
C.~Avart, P.~Komj\'{a}th, T.~\L~uczak, and V.~R\"{o}dl.
\newblock Colorful flowers.
\newblock {\em Topology Appl.}, 156(7):1386--1395, 2009.

\bibitem{Jaroslaw}
Bartłomiej Bosek, Andrzej Dudek, Michał Farnik, Jarosław Grytczuk, and
  Przemysław Mazur.
\newblock Hat chromatic number of graphs, 2019.

\bibitem{Bowler}
N.~Bowler, J.~Erde, F.~Lehner, and M.~Pitz.
\newblock Bounding the cop number of a graph by its genus.
\newblock {\em Acta Math. Univ. Comenian. (N.S.)}, 88(3):507--510, 2019.

\bibitem{Butler}
Steve Butler, Mohammad~T. Hajiaghayi, Robert~D. Kleinberg, and Tom Leighton.
\newblock Hat guessing games.
\newblock {\em SIAM J. Discrete Math.}, 22(2):592--605, 2008.

\bibitem{Erdos1964}
P.~Erd\H{o}s.
\newblock On extremal problems of graphs and generalized graphs.
\newblock {\em Israel J. Math.}, 2:183--190, 1964.

\bibitem{Farnik}
{Micha\l} Farnik.
\newblock {\em A hat guessing game}.
\newblock PhD thesis, Jagiellonian University, 2015.

\bibitem{Feige}
Uriel Feige.
\newblock {You can leave your hat on (if you guess its color)}.
\newblock Technical report, Weizmann Institute, 2004.

\bibitem{Furedi}
Zolt\'{a}n F\"{u}redi and Mikl\'{o}s Simonovits.
\newblock The history of degenerate (bipartite) extremal graph problems.
\newblock In {\em Erd\"{o}s centennial}, volume~25 of {\em Bolyai Soc. Math.
  Stud.}, pages 169--264. J\'{a}nos Bolyai Math. Soc., Budapest, 2013.

\bibitem{GadouleauHats}
Maximilien Gadouleau and Nicholas Georgiou.
\newblock New constructions and bounds for {W}inkler's hat game.
\newblock {\em SIAM J. Discrete Math.}, 29(2):823--834, 2015.

\bibitem{HeWindmill}
Xiaoyu He, Yuzu Ido, and Benjamin Przybocki.
\newblock Hat guessing on books and windmills, 2020.

\bibitem{HeDegeneracy}
Xiaoyu He and Ray Li.
\newblock Hat guessing numbers of degenerate graphs.
\newblock {\em Electron. J. Combin.}, 27(3):Paper No. 3.58, 8, 2020.

\bibitem{Joret}
Gwena\"{e}l Joret, Marcin Kami\'{n}ski, and Dirk~Oliver Theis.
\newblock The cops and robber game on graphs with forbidden (induced)
  subgraphs.
\newblock {\em Contrib. Discrete Math.}, 5(2):40--51, 2010.

\bibitem{KokhasCliquesAndConstructorsI}
K.~P. Kokhas and A.~S. Latyshev.
\newblock Cliques and constructors in the ``hats'' game. {I}.
\newblock {\em Zap. Nauchn. Sem. S.-Peterburg. Otdel. Mat. Inst. Steklov.
  (POMI)}, 488(Kombinatorika i Teoriya Grafov. XI):66--96, 2019.

\bibitem{Kokhas2021}
Aleksei Latyshev and Konstantin Kokhas.
\newblock The hats game. the power of constructors, 2021.

\bibitem{MoharThomassen}
Bojan Mohar and Carsten Thomassen.
\newblock {\em Graphs on Surfaces}.
\newblock Johns Hopkins series in the mathematical sciences. Johns Hopkins
  University Press, 2001.

\bibitem{LLL}
Joel Spencer.
\newblock Asymptotic lower bounds for {R}amsey functions.
\newblock {\em Discrete Math.}, 20(1):69--76, 1977/78.

\bibitem{Szcz}
Witold Szczechla.
\newblock The three colour hat guessing game on cycle graphs.
\newblock {\em Electron. J. Combin.}, 24(1):Paper No. 1.37, 19, 2017.

\bibitem{Winkler}
P.~Winkler.
\newblock Games people don't play.
\newblock In David Wolfe and Tom Rodgers, editors, {\em Puzzlers' Tribute: A
  Feast for the Mind}, chapter~10, pages 301--313. A K Peters, Baltimore, 2002.

\bibitem{Zhu}
Xuding Zhu.
\newblock The game coloring number of pseudo partial {$k$}-trees.
\newblock {\em Discrete Math.}, 215(1-3):245--262, 2000.

\end{thebibliography}

\section{Appendix}
\begin{proof}[Proof of Lemma \ref{lem:Erdos1}]
We induct on $r$. When $r = 2$, we must show that if $\mathcal H$ is a balanced bipartite graph on $2n$ vertices containing at least $3n^{2 - \frac{1}{\ell}}$ edges, then $\mathcal H$ contains a copy of $K_{\ell,\ell}$. By a classical theorem of K\H{o}v\'ari, S\'os, and Tur\'an, $\mathcal H$ contains a copy of $K_{\ell,\ell}$ as long as $|E(\mathcal H)| \geq (\ell - 1)^{1/\ell} (n - \ell + 1) n^{1 - 1/\ell} + (\ell - 1)n$. To show that $3n^{2 - \frac{1}{\ell}}$ is greater than this lower bound, we begin with the following inequality, which can easily be verified graphically:
\[\frac{3}{2} \ell^{1 - \frac{1}{\ell}} - \ell + 1 > 0.\]
Now, since $(\ell - 1)^{1 / \ell} < \frac{3}{2}$ for all $\ell$, and since $n \geq \ell$, we have
\[(3 - (\ell - 1)^{1 / \ell}) n^{1 - \frac{1}{\ell}} - \ell + 1 > 0.\]
Next, since $n > 0$, we have 
\[3n^{2 - \frac{1}{\ell}} - (\ell - 1)^{1 / \ell} n^{2 - \frac{1}{\ell}} - (\ell - 1)n > - (\ell -1)^{\frac{\ell + 1}{\ell}} n^{1 - \frac{1}{\ell}}.\]
Rearranging this equation gives us 
\[3n^{2 - \frac{1}{\ell}} >(\ell - 1)^{1 / \ell}(n - \ell + 1) n^{1 - \frac{1}{\ell}} - (\ell - 1)n,\]
which is exactly what we need to finish the base case.
Next, suppose that $r \geq 3$. We will need to borrow a lemma from the original proof of Erd\H{o}s.
\begin{lemma}\cite{Erdos1964}
\label{lem:Erdos}
Let $S = \{y_1, \dots, y_N\}$ be a set of $N$ elements, and let $A_1, \dots, A_n$ be subsets of $S$. Let $w > 0$, and assume that $\sum_{i = 1}^n |A_i| \geq \frac{nN}{w}$. If $n \geq 2 \ell^2 w^{\ell}$, then there exist $\ell$ distinct sets $A_{i_1}, \dots, A_{i_{\ell}}$ such that $|A_{i_1} \cap \dots \cap A_{i_{\ell}}| \geq \frac{N}{2w^{\ell}}$.
\end{lemma}
Now, suppose we have a balanced $r$-partite $r$-graph $\mathcal H$ with $rn$ vertices and $t \geq 3n^{r - \frac{1}{\ell^{r-1}}}$ edges. We choose one of the $r$ partite sets of $\mathcal H$ and name its vertices $x_1, \dots, x_n$. Next, we set $N = n^{r-1}$, and we let $y_1, \dots, y_N$ denote the set of $(r-1)$-tuples of vertices that can be obtained by choosing exactly one vertex from each partite set of $\mathcal H$ outside of $\{x_1, \dots, x_r\}$. Then, for each $x_i$, we let $A_i$ contain those $y_j$ for which $x_i \cup y_j \in E(\mathcal H)$. We have 
\[\sum_{i = 1}^n |A_i| = t \geq 3n^{r - \frac{1}{\ell^{r-1}}}.\]
We set $w = \frac{1}{3} n^{\frac{1}{\ell^{r-1}}},$ and then it is easy to verify that $t \geq \frac{nN}{w}$ and that $n \geq 2 \ell^2 w^{\ell}$, so the hypotheses of Lemma \ref{lem:Erdos} hold. Hence, we may choose $\ell$ vertices $x_{i_1}, \dots, x_{i_{\ell}}$ whose neighborhoods intersect in at least 
\[\frac{N}{2w^{\ell}} = \frac{3^{\ell}}{2} n^{r - \frac{1}{\ell^{r-2}}} > 3  n^{r - \frac{1}{\ell^{r-2}}}\]
$(r-1)$-tuples. Then, by the induction hypothesis, we may find a copy of $K^{(r-1)}_{\ell}$ among these $(r-1)$-tuples, and this $K^{(r-1)}_{\ell}$ along with the vertices $x_{i_1}, \dots, x_{i_{\ell}}$ form a copy of $K^{(r)}_{\ell}$. This completes the proof.
\end{proof}
\begin{proof}[Proof of estimates in Theorem \ref{thm:planar}]
Recall that $\ell_5 = 20s$ and $s_4 = (20)^6s^7$. We will use the inequalities
\begin{eqnarray}
 \label{eq:1} (3s^2 + 3s + 3)^{3(s^2+s+1)^2}+1 &<& 2^{(3s)^{5}} \\
\label{eq:2} (3 (s+1)^6 + 3(s+1)^3 + 3)^{3((s+1)^6 + (s+1)^3 + 1)^2} +1&<& 2^{(3s)^{13}}
\end{eqnarray}
for $s \geq 1$.
From (\ref{eq:1}), we see that 
\[\log_2 \ell_4 < (3s_4)^5 = 3^5 2^{60} 5^{30} s^{35} < 2^{138}s^{35}.\]
Then, 
\[\log_2 s_3 = \log s + 6 \log_2 \ell_4 < 2^{141}s^{35}.\]
Then, using (\ref{eq:2}), 
$\log_2 \ell_3 < ( 3s_3)^{13}$, and so 
\[\log_2 \log_2 \ell_3 < 13 \log_2 3 + 13 \log_2 s_3 < 2^{145} s^{35}.\]
The remaining bounds can be proven similarly using (\ref{eq:2}).
\end{proof}

\end{document}